\input{style/arxiv-general.cfg}
\documentclass[MSNbibl,number,citesort,seceqn,dvips]{arxbj}
\makeatletter
   \@ifpackageloaded{graphicx}{}{\usepackage{graphicx}}
\makeatother
\usepackage{mathrsfs,upgreek}
\usepackage{graphicx}

\aid{0}
\volume{21}
\issue{2}
\pubyear{2015}
\firstpage{851}
\lastpage{880}
\doi{10.3150/13-BEJ590} 

\makeatletter

\newcommand{\rrVert}{\Vert}
\newcommand{\rrvert}{\vert}
\newcommand{\llVert}{\Vert}
\newcommand{\llvert}{\vert}
\newtheorem{theorem}{Theorem}[section]
\newremark{remark}{Remark}[section]
\newremark{Remarks}{Remarks}
\newtheorem{proposition}[theorem]{Proposition}
\newtheorem{lemma}[theorem]{Lemma}
\newcommand{\ww}{\tilde}

\def\mC{\mathcal}
\newcommand{\eps}{\varepsilon}
\newcommand{\D}{\Delta}
\newcommand{\dtv}{d_{\mathrm{TV}}}
\newcommand{\dw}{d_{\mathrm{W}}}
\newcommand{\dk}{d_{\mathrm{K}}}

\newcommand{\bigo}{\mathrm{O}}
\newcommand{\lito}{\mathrm{o}}
\newcommand{\dloc}{d_{\mathrm{loc}}}
\newcommand{\toinf}{\to\infty}
\newcommand{\I}{\mathrm{I}}
\newcommand{\Po}{\mathrm{Po}}
\newcommand{\TP}{\mathrm{TP}}
\newcommand{\Bi}{\operatorname{Bi}}
\newcommand{\Be}{\operatorname{Be}}
\newcommand{\N}{\mathrm{N}}
\newcommand{\IE}{\mathbb{E}}
\newcommand{\IP}{\mathbb{P}}
\newcommand{\Var}{\operatorname{Var}}
\newcommand{\Cov}{\operatorname{Cov}}
\newcommand{\sgn}{\operatorname{sgn}}
\newcommand{\law}{\mathscr{L}}
\newcommand{\IZ}{\mathbb{Z}}
\newcommand{\IR}{\mathbb{R}}

\newcommand{\eqref}[1]{(\ref{#1})}
\newcommand{\binom}[2]{{#1 \choose #2}}

\def\mysigma{{{s}}}

\def\vfrac#1#2{(#1)/#2}

\def\sklfrac#1#2{(#1/#2)}

\makeatother

\begin{document}
\begin{frontmatter}

\title{Local limit theorems via Landau--Kolmogorov inequalities}
\runtitle{LLTs via Landau--Kolmogorov inequalities}

\begin{aug}
\author[A]{\inits{A.}\fnms{Adrian} \snm{R\"ollin}\thanksref{A}\ead[label=e1]{adrian.roellin@nus.edu.sg}\ead[label=u1,url]{www.stat.nus.edu.sg/\textasciitilde staar}}
\and
\author[B]{\inits{N.}\fnms{Nathan} \snm{Ross}\corref{}\thanksref{B}\ead[label=e2]{nathan.ross@unimelb.edu.au}\ead[label=u2,url]{www.ms.unimelb.edu.au/\textasciitilde rossn1}}
\address[A]{Department of Statistics and Applied Probability,
National University of Singapore,
6 Science Drive 2,
Singapore 117546.
\printead{e1,u1}}
\address[B]{Department of Mathematics and Statistics,
University of Melbourne, VIC 3010, Australia.\\
\printead{e2,u2}}
\end{aug}

\received{\smonth{5} \syear{2013}}
\revised{\smonth{12} \syear{2013}}

%
\begin{abstract}
In this article, we prove new inequalities between some common probability
metrics. Using these inequalities, we obtain novel local limit theorems for
the magnetization in the Curie--Weiss model at high temperature, the
number of
triangles and isolated vertices in Erd\H{o}s--R\'enyi random graphs, as
well as the
independence number in a geometric random graph. We also give upper
bounds on
the rates of convergence for these local limit theorems and also for some
other probability metrics. Our proofs are based on the Landau--Kolmogorov
inequalities and new smoothing techniques.
\end{abstract}

%

\end{frontmatter}

\section{Introduction}\label{secintro}

If two probability distributions are close in some metric are they also close
in other stronger or different metrics?
General inequalities between many common probability metrics are known; see
for example Gibbs and Su \cite{Gibbs2002} for a compilation
of such results. But
one may
wonder if it is possible to sharpen such inequalities
by imposing simple conditions on the distributions
under consideration. An early attempt in this direction was made by
McDonald \cite{McDonald1979}, who was able to deduce
a local limit theorem for
sums of
integer valued random variables from a central limit theorem by
imposing an additional ``smoothness'' condition on the distribution of
the sum. In this article, we take this
approach much further by providing general inequalities between some
common probability metrics with integer support that contain an additional
factor that measures the ``smoothness'' of the distributions
under consideration; the smaller this factor is, the better the bounds obtained.

To state a simple version of our main result, we need some basic notation.
For a function $f$ with domain the
integers, denote for $1\leq p < \infty$,
\begin{eqnarray*}
\llVert f\rrVert _p= \biggl(\sum_{i\in\IZ}
\bigl\llvert f(i)\bigr\rrvert ^p \biggr)^{1/p}
\end{eqnarray*}
and $\llVert  f\rrVert _\infty=\sup_{i\in\IZ} \llvert
f(i)\rrvert $, and
also define the operators $\D^n$ recursively by
\begin{eqnarray*}
\D^0 f(k) =f(k) \quad \mbox{and}\quad  \D^{n+1} f(k)=
\D^{n}f(k+1)-\D^n f(k).
\end{eqnarray*}

A consequence of our main
theoretical result, Theorem~\ref{thm2} below, is that if $F$ and $G$ are
distribution functions of integer supported distributions, then
for some universal constant $C$,
%
\begin{equation}
\label{1} \llVert \D F-\D G\rrVert _\infty\leq C \llVert F-G\rrVert
_\infty^{1/2} \bigl( \bigl\llVert \D^3 F\bigr
\rrVert _1+\bigl\llVert \D^3 G\bigr\rrVert _1
\bigr)^{1/2}.
\end{equation}
Here $(\llVert \D^3 F\rrVert _1+\llVert \D^3 G\rrVert _1)^{1/2}$ is the smoothing factor
referred to above, so the inequality says that if we can bound it and the
supremum of the pointwise differences of the distribution functions $F$ and
$G$ (called the \emph{uniform} or \emph{Kolmogorov metric}), then we
have a
bound on the left-hand side of \eqref{1}, the supremum of the
differences of
point probabilities; the latter is a quantity that will allow us to obtain
local limit theorems.

In practice, it may appear difficult to obtain bounds on the smoothing term
since it is defined in terms of quantities we wish to study. In this article,
we think of $F$ as being a complicated distribution of interest (e.g., the
number of triangles in a random graph model) and of $G$ as a well-known
distribution which we are using to approximate $F$ (e.g., a discretized
normal or a translated Poisson distribution). Thus, bounding $\llVert \D
^3 G\rrVert _1$
should not be difficult -- we provide what is needed for our theory and
applications in Lemma~\ref{lem7} below -- so the only real difficulty in
using \eqref{1} in application is bounding $\llVert \D^3 F\rrVert _1$ and in
Section~\ref{sec2} we develop tools for this purpose.

To get a sense of the style of result we aim to achieve, we apply
\eqref{1}
in the setting of the approximation of the binomial distribution by the
normal, where much is known.

\subsection*{Binomial local limit theorem}
Let $X\sim\Bi(n,p)$ and let $Y$ have a discretized normal
distribution with
mean $\mu:=np$ and variance $\sigma^2:=np(1-p)$, that is,
%
\begin{equation}
\label{2} \IP(Y=k)=\frac{1}{\sqrt{2\uppi}}\int_{(k-1/2-\mu)/\sigma
}^{(k+1/2-\mu)/\sigma}
\mathrm{e}^{-u^2/2}\,\mathrm{d}u.
\end{equation}
If $F$ and $G$ are the distribution functions of $X$ and $Y$, then it
is well
known
that $\llVert  F-G\rrVert _\infty\asymp\sigma^{-1} \asymp
n^{-1/2}$; here and
below the limits and asymptotics are as $n\to\infty$.
Also $\D^3 G(k)=\D^2 \IP(Y=k)$ and some basic calculus and \eqref{2}
imply that
$\llVert \D^2 \IP(Y=\cdot)\rrVert _1\asymp\sigma
^{-2}\asymp n^{-1}$.
Due to the closeness of the binomial distribution to the normal, we anticipate
$\llVert \D^3 F\rrVert _1$ to be of this same order as
$\llVert \D^3 G\rrVert _1$ and
in fact
Proposition~\ref{prop1} below bounds this term as $\llVert \D^3
F\rrVert _1=\bigo(\sigma^{-2})$.
Putting this all into \eqref{1}, we have that
\begin{eqnarray*}
\llVert \D F-\D G\rrVert _\infty=\bigo\bigl(\sigma ^{-3/2}
\bigr)=\bigo\bigl(n^{-3/4}\bigr).
\end{eqnarray*}
In fact, it is well known that
\begin{eqnarray*}
\llVert \D F-\D G\rrVert _\infty\asymp\sigma^{-2} \asymp
n^{-1}.
\end{eqnarray*}
This example illustrates that we do not expect our approach to yield tight
rates in application. However, our purpose here is to provide a
method that can be applied to yield new convergence results where
little is
known and, as a by-product of our method of proof, to give some upper
bounds on
the rates of convergence. We emphasize that apart from well known results
about sums of independent random variables,
rates of convergence in local limit theorems are not common in the literature:
such results are typically difficult to obtain.
To the best of our knowledge, all of our results and upper bounds on
rates are
new. Outside of a few remarks we will not address the interesting
but more theoretical question of the optimality of the bounds obtained -- we
shall focus on applications.

The remainder of the paper is organized as follows. In Section~\ref
{sec1}, we
prove our main theoretical results, inequalities of the form \eqref
{1}; these
will follow from discrete versions of the classical Landau--Kolmogorv
inequalities. In Section~\ref{sec2}, we
develop tools to bound $\llVert \D^3F\rrVert _1$ and the
analogous quantities appearing
on the right hand side of generalizations of \eqref{1}. In
Section~\ref{sec3},
we illustrate our approach in a few applications, in particular we
obtain new
local limit theorems with bounds on the rates of convergence for the
magnetization in the Curie--Weiss model, the number of isolated vertices and
triangles in Erd\H{o}s--R\'enyi random graphs and the independence
number of a geometric
random graph. We also obtain other new limit theorems and bounds on
rates for
some of these applications.

\section{Main result}\label{sec1}

Our main theoretical result
is easily derived from a discrete version of
the classical Landau inequality (see Hardy, Landau and Littlewood \cite{Hardy1935},
Section~3)
which relates the
norm of a function with that of its first and second derivatives. There
are many extensions and embellishments of this inequality
in the analysis literature; see
Kwong and Zettl \cite{Kwong1992} for a book length treatment.

\begin{theorem}[({Kwong and Zettl \cite{Kwong1992}, Theorem~4.1})]\label{thm1}
Let $k$ and $n$ be integers with $1\leq k < n$ and let $1\leq p,
q,r\leq
\infty$ given. There is a positive number $C:=C(n,k,p,q,r)$ such that
%
\begin{equation}
\label{2b} \bigl\llVert \D^k f\bigr\rrVert _q \leq C
\llVert f\rrVert ^{\alpha}_p \bigl\llVert \D^n f
\bigr\rrVert ^{\beta}_r
\end{equation}
for all $f\dvtx \IZ\to\IR$ with $\llVert  f\rrVert _p<\infty$
and $\llVert \D^n
f\rrVert _r<\infty$, if and only if
\begin{eqnarray*}
\frac{n}{q} \leq\frac{n-k}{p} + \frac{k}{r},
\end{eqnarray*}
$\alpha=1-\beta$ and
\begin{eqnarray*}
\beta= \frac{k-1/q+1/p}{n - 1/r + 1/p}.
\end{eqnarray*}
\end{theorem}

\begin{remark}
Much of the literature surrounding these inequalities is concerned with
finding the optimal value of the constant $C$. In the case that $n=2$ and
either $p=q=r=1$ or $n=3$, $p=q=\infty$, and $r=1$, we can take
$C=\sqrt{2}$;
see Kwong and Zettl \cite{Kwong1992}, Theorem~4.2. These are two of
the main cases discussed
below. Also, an inductive argument in $n$ implies that in the former case
above we may take $C=2^{(n-1)/2}$ for $n\geq2$ and in the latter
$C=2^{(n-2)/2}$ for $n\geq3$. These facts are not critical in what
follows, so
for the sake of simplicity we do not discuss such constants in further detail.
\end{remark}

The key connection between Theorem~\ref{thm1} and
what will follow is that if $F$
and $G$ are distribution functions of integer supported
probability distributions, then some well-known probability metrics can be
expressed as
\begin{eqnarray*}
\begin{array}{rcl@{\qquad}l}
\dk(F,G) &=& \llVert F-G\rrVert _\infty& \mbox{(Kolmogorov
metric),}
\\
\dw(F,G) &=&\llVert F-G\rrVert _1& \mbox{(Wasserstein
metric),}
\\
\dloc(F,G) &=&\llVert \D F-\D G\rrVert _\infty& \mbox {(local
metric),}
\\
\dtv(F,G) &=&{\tfrac{1}{2}}\llVert \D F-\D G\rrVert _1 & \mbox{(total variation metric)}.
\end{array}
\end{eqnarray*}
Note that $\dloc(F,G)$ is the supremum of point probabilities between the
distributions
given by $F$ and $G$ and is the appropriate metric to use to show local limit
theorems.

We are now in a position to state our main theoretical result, but first
a last bit of notation. Let $F$
be a distribution function with support on $\IZ$. If $m$ is a
positive integer, define
\begin{eqnarray*}
\bar{F}^m (j) = \frac{\bar{F}(j)+\cdots+F(j-m+1)}{m};
\end{eqnarray*}
this is the distribution function of the convolution of $F$ with the
uniform distribution on $\{0,\dots,m-1\}$. Note that $\bar{F}^1 = F$ and
that if the integer valued random variable $X$ has distribution
function $F$,
then
%
\begin{equation}
\label{3} \D \bar{F}^m(j)=\frac{1}{m} \IP(j-m+1<X\leq j+1).
\end{equation}

%
\begin{theorem}\label{thm2} If $l\geq1$ and $m\geq1$ are integers,
then there
is a constant
$C>0$ such that, for all distribution functions $F$ and $G$
of integer supported probability distributions,
\begin{eqnarray*}
d_1\bigl(\bar{F}^m,\bar{G}^m\bigr) \leq C
d_2(F,G)^{1-\beta} \bigl(\bigl\llVert \D^{l+1}
\bar{F}^m\bigr\rrVert _1+\bigl\llVert \D^{l+1}
\bar{G}^m\bigr\rrVert _1\bigr)^{\beta}
\end{eqnarray*}
for the following combinations of $d_1$, $d_2$ and $\beta$:
%
\begin{equation}
\label{4} %
\begin{tabular}{@{}llll@{}}
\hline
&$d_1$ &
$d_2$ & $\beta$
\\
\hline
\mbox{\textup{(i)}}&$\dloc$& $\dtv$& $1/l$
\\
\mbox{\textup{(ii)}}&$\dloc$& $\dk$& $1/l$
\\
\mbox{\textup{(iii)}}&$\dloc$& $\dw$& $2/(l+1)$
\\
\mbox{\textup{(iv)}}&$\dtv$& $\dw$& $1/(l+1)$
\\
\mbox{\textup{(v)}}&$\dk$& $\dw$& $1/(l+1)$
\\
\hline
\end{tabular} %
\end{equation}
\end{theorem}

\begin{pf}
To prove (ii)--(iv), apply Theorem~\ref{thm1} to the function
$\bar{F}^m-\bar{G}^m$,
with $n=l+1$, $k=r=1$ and the following values of $p$ and $q$:
\begin{eqnarray*}
\mbox{(ii)}\enskip q=\infty,p=\infty, \qquad \mbox{(iii)}\enskip q=\infty,p=1, \qquad \mbox{(iv)}\enskip q=1,p=1;
\end{eqnarray*}
then use $d_2(\bar{F}^m,\bar{G}^m)\leq d_2(F, G)$ and the triangle inequality. For
(i) and (v)
use (ii) and (iv), respectively, and then use the fact that $\dk
\leq\dtv$.
\end{pf}

\begin{Remarks*}
\begin{enumerate}[3.]
\item[1.] We mainly use
Theorem~\ref{thm2} with $m=1$, where its meaning is most
transparent.
For $m>1$, the following direct consequence of \eqref{3}
shows that $\bar{F}^m$ can be used to prove
local limit theorems for ``clumped'' probabilities
where the corresponding pointwise results
may not hold.

%
\begin{lemma}\label{lem1}
If $X$ and $Y$ are integer valued random variables with respective
distribution
functions $F$ and $G$, then
\begin{eqnarray*}
\sup_{k\in\IZ}\bigl\llvert \IP(k<X\leq k+m)-\IP(k<Y\leq k+m)\bigr
\rrvert = m \dloc\bigl(\bar{F}^m, \bar{G}^m\bigr).
\end{eqnarray*}
\end{lemma}

\item[2.] Theorem~\ref{thm2} is really a special case of Theorem~\ref
{thm1} with
$k=r=1$, but it is clear that similar statements hold by applying
Theorem~\ref{thm1} to other values of $k$ and $r$. We choose the value $r=1$
because we are able to bound $\llVert \D^n F\rrVert _1$.
Using the obvious inequality
$\llVert \D^n F\rrVert _\infty\leq\llVert \D^{n+1}
F\rrVert _1$, we could also
usefully apply
Theorem~\ref{thm1} with $r=\infty$, but this change has no effect on the
value of $\beta$ for a given $k$, $q$ and $p$. The term $\llVert
\D^2
F\rrVert _\infty$
also appears in the local limit theorem results of McDonald \cite{McDonald1979} and
Davis and McDonald \cite{Davis1995}. However, the crucial
advantage of $\llVert \D^n
F\rrVert _1$ over
$\llVert \D^n F\rrVert _\infty$ is that the former is -- as
we will
show -- amenable to
bounds via probabilistic techniques, whereas the latter seems difficult to
handle directly.
\item[3.] Inequality \eqref{2b} cannot be improved in general, but since we
are considering
such inequalities only over the class
of functions that are the difference of two distribution functions, it is
possible that Theorem~\ref{thm2} could be sharpened, either in
increasing the
exponents or decreasing the constants. Also note that using the triangle
inequality in Theorem~\ref{thm2} causes some loss of sharpness, but we gain
the ability to bound the terms appearing in application, which is our main
focus.
\end{enumerate}
\end{Remarks*}

\section{Estimating the measure of smoothness}\label{sec2}

In this section, we develop techniques to bound $\llVert \D
^n\bar{F}^m\rrVert _1$.
Our main tools are Theorems \ref{thm3} and \ref{thm4} below but
first we state some simple results.
To lighten
the notation somewhat, write
\begin{eqnarray*}
D_{n,m}(F) := m \bigl\llVert \D^{n+1} \bar{F}^m\bigr
\rrVert _1,
\end{eqnarray*}
or for a random variable $W$ with distribution function $F$,
write $D_{n,m}(\law(W))$ for $D_{n,m}(F)$,
and $D_n$ for $D_{n,1}$. Furthermore, define recursively the difference
operators
\begin{eqnarray*}
\D^n_m F(j) = \D^{n-1}_m F(j+m)-
\D^{n-1}_m F(j),
\end{eqnarray*}
where $\D^0_m F(j) = F(j)$.

Note that for a random variable $W$,
\begin{eqnarray*}
D_1\bigl(\law(W)\bigr)=2\dtv\bigl(\law(W),\law(W+1)\bigr).
\end{eqnarray*}
By a well-known representation of the total variation distance (see, for
example,
Gibbs and Su \cite{Gibbs2002}), we have
%
\begin{equation}
\label{5} D_1\bigl(\law(W)\bigr)=\sup_{\llVert  g\rrVert _\infty\leq1}
\IE\D g(W).
\end{equation}
Some of our techniques below are extensions of those for bounding the quantity
on the right-hand side \eqref{5}, and so we use the following generalization
of \eqref{5} in bounding $D_{n,m}(F)$.

%
\begin{lemma}\label{lem2} Let $n$ and $m$ be nonnegative integers. Let
$W$ be
a random variable with integer support. Then
\begin{eqnarray*}
D_{n,m}\bigl(\law(W)\bigr) = \sup_{\llVert  g\rrVert _\infty\leq
1}\IE
\D_m^n g(W).
\end{eqnarray*}
\end{lemma}

\begin{pf} We only show the case $n=1$; the general case is similar.
Denote $p_i=\IP[W=i]$ for $i\in\IZ$. We have
\begin{eqnarray*}
\IE\D_m g(W) & =& \sum_{i\in\IZ}p_i
\bigl(g(i+m)-g(i) \bigr)
\\
& =& \sum_{i\in\IZ}(p_{i-m}-p_i)
g(i)
\\
& =& m\sum_{i\in\IZ} \biggl(\frac{p_{i-m}+\cdots+p_{i-1}}{m} -
\frac{p_{i-m+1}+\cdots+p_i}{m} \biggr) g(i)
\\
& =& m\sum_{i\in\IZ} \bigl(\D \bar{F}^{m}(i-2)-\D
\bar{F}^{m}(i-1) \bigr) g(i)
\\
& =& -m\sum_{i\in\IZ}\D^2\bar{F}^{m}(i-2)
g(i),
\end{eqnarray*}
where in the fourth equality we have used \eqref{3}.
We see that for all $g$ such that $\llVert  g\rrVert _\infty
\leq1$, $ \IE\D
_m g(W)
\leq
m \llVert \D^{2} \bar{F}^m\rrVert _1 $, and choosing $g(i) = -\sgn
\D^2\bar{F}^{m}(i-2)$
implies
the claim.
\end{pf}

The following sequence of lemmas provide tools for bounding
$D_{n,m}(F)$. The
proofs are mostly straightforward. We assume that all random variables are
integer valued.

%
\begin{lemma}\label{lem3} Let $n$ and $m$ be positive integers.
If $W$ is a random variable, then
\begin{eqnarray*}
D_{n,m}\bigl(\law(W)\bigr)\leq m D_{n,1}\bigl(\law(W)\bigr).
\end{eqnarray*}
\end{lemma}

\begin{pf}
If $W$ has distribution function $F$, then the triangle inequality implies
\begin{eqnarray*}
D_{n,m}\bigl(\law(W)\bigr)&=&\sum_{k\in\IZ}
\bigl\llvert \D^{n+1} F(k)+\cdots+\D^{n+1} F(k+m-1)\bigr\rrvert
\\
&\leq&\sum_{k\in\IZ}\bigl\llvert \D^{n+1} F(k)
\bigr\rrvert +\cdots +\sum_{k\in\IZ}\bigl\llvert
\D^{n+1} F(k+m-1)\bigr\rrvert
\\
&=&mD_{n,1}(F) = m D_{n,1}\bigl(\law(W)\bigr).
\end{eqnarray*}
\upqed
\end{pf}

\begin{lemma}\label{lem4} Let $n$ and $m$ be positive integers. If $W$
is a
random variable and $\mC{F}$ is a $\sigma$-algebra, then
\begin{eqnarray*}
D_{n,m}\bigl(\law(W)\bigr) \leq\IE D_{n,m} \bigl(
\law(W|\mC{F})\bigr).
\end{eqnarray*}
\end{lemma}

\begin{pf} If $f$ is a bounded function, then
\begin{eqnarray*}
\bigl\llvert \IE\D^n_m f(W)\bigr\rrvert \leq\IE\bigl
\llvert \IE\bigl[\D ^n_m f(W)|\mC{F}\bigr]\bigr
\rrvert \leq \llVert f\rrVert _\infty\IE D_{n,m} \bigl(
\law(W|\mC{F})\bigr).
\end{eqnarray*}
By Lemma~\ref{lem2}, the claim follows.
\end{pf}

\begin{lemma}\label{lem5} If $X_1$ and $X_2$ are independent random variables,
then, for all $n_1,n_2,m\geq1$,
%
\begin{equation}
\label{6} D_{n_1+n_2,m}\bigl(\law(X_1+X_2)
\bigr)\leq D_{n_1,m}\bigl(\law(X_1)\bigr)D_{n_2,m}
\bigl(\law(X_2)\bigr).
\end{equation}
If $X_1,\dots,X_N$ is a sequence of independent random
variables and $n\leq N$,
%
\begin{equation}
\label{7} D_{n,m}\bigl(\law(X_1+\cdots+X_N)
\bigr) \leq\prod_{i=1}^n D_{1,m}
\bigl(\law(X_i)\bigr).
\end{equation}
\end{lemma}

\begin{pf} Let $f$ be a bounded function and
define
\begin{eqnarray*}
g(x) := \IE\D^{n_2}_m f(x+X_2)=\sum
_{j\in\IZ} \D^{n_2}_m f(x+j)
\IP(X_2=j).
\end{eqnarray*}
Note that $\llVert  g\rrVert _\infty\leq D_{n_2,m}(X_2)\llVert  f\rrVert _\infty$
and we claim
\begin{eqnarray*}
\IE\D^{n_1+n_2}_m f(X_1+X_2) = \IE
\D^{n_1}_m g(X_1),
\end{eqnarray*}
which follows by independence (that is, the conditioning has no effect).
Hence,
\begin{eqnarray*}
D_{n_1+n_2,m}\bigl(\law(X_1+X_2)\bigr)\leq
D_{n_1,m}(X_1)\llVert g\rrVert _\infty\leq
D_{n_1,m}(X_1)D_{n_2,m}(X_2)\llVert f
\rrVert _\infty
\end{eqnarray*}
which proves \eqref{6}.
A similar argument establishes that
$D_{n,m}(X_1+X_2)\leq D_{n,m}(X_1)$ so now \eqref{7} follows by induction.
\end{pf}

The quantity $D_1(\law(W),\law(W+1))=2\dtv(W,W+1)$ has appeared in extending
the central limit theorem for sums of integer valued random variables to
stronger metrics such as the total variation and local limit metric;
see, for
example, Barbour and Xia \cite{Barbour1999}, Barbour and {\v{C}}ekanavi{\v {c}}ius \cite{Barbour2002} and Goldstein and Xia \cite
{Goldstein2006}. In
these cases, the main tool for bounding $D_1(\law(W))$ was initially the
Mineka coupling but the following result is now the best available (see
P{\'o}sfai \cite{Posfai2009} and references there).

%
\begin{lemma}[(Mattner and Roos \cite{Mattner2007}, Corollary~1.6)]\label{lem6} Let
$X_1,X_2,\ldots,X_N$ be a
sequence of independent integer valued random variables and $S_N =
\sum_{i=1}^N X_i$. Then
\begin{eqnarray*}
D_{1}\bigl(\law(S_N)\bigr) = D_{1,1}\bigl(
\law(S_N)\bigr) \leq \sqrt{\frac{8}{\uppi}} \Biggl(
\frac{1}{4} +\sum_{i=1}^N
\biggl(1-{\frac{1}{2}}D_{1}\bigl(\law(X_i)\bigr)
\biggr) \Biggr)^{-1/2}.
\end{eqnarray*}
\end{lemma}

The following two theorems are our main contributions in this section.
To illustrate their use, we apply them in a simple
setting in Proposition~\ref{prop1} at the end of this section.

%
\begin{theorem}\label{thm3}
Let $(X,X')$ be an exchangeable pair and let $W := W(X)$ and $W' := W(X')$
take
values on the integers.
Define
\begin{eqnarray*}
Q_{m}(x) = \IP\bigl[W' = W+m|X=x\bigr]
\end{eqnarray*}
and $q_m = \IE Q_m(X) = \IP[W' = W+m]$.
Then, for every positive integer $m$,
\begin{eqnarray*}
D_{1,m}\bigl(\law(W)\bigr) \leq\frac{\sqrt{\Var Q_m(X)}+\sqrt{\Var Q_{-m}(X)}}{q_m}.
\end{eqnarray*}
\end{theorem}

\begin{pf}
To prove the first assertion, we must bound $\llvert \IE\Delta
_mg(W)\rrvert $
for all
$g$
with norm no greater than one. To this end, exchangeability implies
that for
all
bounded functions $g$
%
\begin{eqnarray}
\label{8} 0&=&q_m^{-1}\IE \bigl\{\I
\bigl[W'=W+m\bigr]g\bigl(W'\bigr)-\I
\bigl[W'=W-m\bigr]g(W) \bigr\}\nonumber
\\[-8pt]\\[-8pt]
&=&q_m^{-1}\IE \bigl\{Q_m(X)g(W+m)-Q_{-m}(X)g(W)
\bigr\} ,\nonumber
\end{eqnarray}
so that
%
\begin{eqnarray}
\label{9} && \bigl\llvert \IE\Delta_m g(W)\bigr\rrvert
\nonumber\\
&&\quad  =\bigl\llvert \IE \bigl\{ \bigl(1-q_m^{-1}Q_m(X)
\bigr)g(W+m)-\bigl(1-q_m^{-1}Q_{-m}(X)\bigr)g(W)
\bigr\} \bigr\rrvert
\\
&&\quad  \leq\frac{\sqrt{\IE \{g(W+m)^2 \} \Var Q_m(X)}+\sqrt {\IE
\{g(W)^2 \} \Var Q_{-m}(X)}}{q_m},\nonumber
\end{eqnarray}
where in the inequality we use first the triangle inequality and then
Cauchy--Schwarz. Taking the supremum over $g$ with $\llVert  g\rrVert _\infty
\leq1$ in
\eqref{9} proves the theorem.
\end{pf}

Theorem~\ref{thm3} is inspired by Stein's method of exchangeable pairs
as used by Chatterjee, Diaconis and Meckes \cite{Chatterjee2005} and
R{\"o}llin \cite{Rollin2007a}. Our next result
extends and embellishes
Theorem~\ref{thm3}.

%
\begin{theorem}\label{thm4}
Let $(X,X',X'')$ be three consecutive steps of a reversible
Markov chain in equilibrium. Let $W$ and $W'$ be as in Theorem~\ref
{thm3} and,
in addition, $W'' := W(X'')$. Define
\begin{eqnarray*}
Q_{m_1,m_2}(x) = \IP\bigl[W' = W+m_1,W''
= W'+m_2|X=x\bigr].
\end{eqnarray*}
Then, for every positive integer $m$,
\begin{eqnarray*}
D_{2,m}\bigl(\law(W)\bigr) & \leq&\frac{1}{q_m^2} \bigl(2\Var
Q_m(X)+\IE\bigl\llvert Q_{m,m}(X)-Q_m(X)^2
\bigr\rrvert
\\
&&\hphantom{\frac{1}{q_m^2} \bigl(}{} +2\Var Q_{-m}(X)+\IE\bigl\llvert Q_{-m,-m}(X)-Q_{-m}(X)^2
\bigr\rrvert \bigr).
\end{eqnarray*}
\end{theorem}

\begin{pf}
Similar to the proof of Theorem~\ref{thm3}, we want to bound
$\IE\Delta_m^2g(W)$
for all $g$ with norm no greater than one. We begin with the trivial equality
%
\begin{eqnarray}
\label{10}&& \IE \bigl\{\I\bigl[W'=W+m,W''=W'+m
\bigr]g(W+m) \bigr\}
\\
 \label{11}&&\quad  =\IE \bigl\{\I\bigl[W'=W+m,W''=W'+m
\bigr]g\bigl(W'\bigr) \bigr\} .
\end{eqnarray}
Conditioning on $X$ in \eqref{10} and on $X'$ in \eqref{11}, the Markov
property and
reversibility imply
\begin{eqnarray*}
\IE \bigl\{Q_{m,m}(X)g(W+m) \bigr\} =\IE \bigl\{ Q_m(X)Q_{-m}(X)g(W)
\bigr\} ,
\end{eqnarray*}
and similarly
\begin{eqnarray*}
\IE \bigl\{Q_{-m,-m}(X)g(W) \bigr\} = \IE \bigl\{ Q_m(X)Q_{-m}(X)g(W+m)
\bigr\} .
\end{eqnarray*}
Using these two equalities coupled with \eqref{8} we find that for
bounded $g$
\begin{eqnarray*}
0&=&\IE \bigl\{g(W+2m) \bigl(q_m^{-2}Q_{m,m}
(X)-2q_m^{-1}Q_m(X)\bigr) \bigr\}
\\
&&{} -2\IE \bigl\{ g(W+m) \bigl(q_m^{-2}Q_m(X)Q_{-m}(X)-q_m^{-1}Q_m(X)-q_m^{-1}Q_{-m}(X)
\bigr) \bigr\}
\\
&&{} +\IE \bigl\{g(W) \bigl(q_m^{-2}Q_{-m,-m}
(X)-2q_m^{-1}Q_{-m}(X)\bigr) \bigr\} .
\end{eqnarray*}
It is now not hard to see that
\begin{eqnarray*}
 \IE\Delta_m^2 g(W) &=& \IE g(W+2m) -2 \IE g(W+m) + \IE
G(W)
\\
& =&\IE \bigl\{g(W+2m) \bigl( \bigl(1-q_m^{-1}Q_m(X)
\bigr)^2+q_m^{-2} \bigl(Q_{m,m}
(X)-Q_m(X)^2 \bigr) \bigr) \bigr\}
\\
&&{} -2\IE \bigl\{g(W+m) \bigl(1-q_m^{-1}Q_m(X)
\bigr) \bigl(1-q_{m}^{-1}Q_{-m} (X) \bigr) \bigr\}
\\
&&{} +\IE \bigl\{g(W) \bigl( \bigl(1-q_m^{-1}Q_{-m}(X)
\bigr)^2+q_m^{-2} \bigl(Q_{-m,-m}
(X)-Q_ {-m}(X)^2 \bigr) \bigr) \bigr\} .
\end{eqnarray*}
The theorem now follows by taking the supremum over $g$ with norm no greater
than one and applying the triangle inequality and Cauchy--Schwarz.
\end{pf}

To better understand how Theorems \ref{thm3} and \ref{thm4} work in practice,
we derive the following result.

%
\begin{proposition}\label{prop1}
If $W\sim\Bi(n,p)$, then
\begin{eqnarray*}
D_2\bigl(\law(W)\bigr)\leq \frac{1}{n} \biggl(
\frac{2p+1}{1-p}+\frac{2(1-p)+1}{p} \biggr).
\end{eqnarray*}
\end{proposition}

\begin{pf}
Retaining the notation above, we define the following Markov chain
on sequences of zeros and ones of length $n$, reversible with respect to
the Bernoulli product measure. At each step in the chain,
a coordinate is selected uniformly at random and resampled. Let
$X,X',X''$ be
three
consecutive
steps in this chain in stationary and $W(=W(X)),W',W''$ be the number
of ones in these $0-1$ configurations. We find
\begin{eqnarray*}
Q_{1}(X)=\frac{n-W}{n} p \quad \mbox{and}\quad  Q_{-1}(X)=
\frac{W}{n}(1-p),
\end{eqnarray*}
since, for example, in order for the number of ones to increase by one from
$X$, a zero must be selected
(with probability $(n-W)/n$) and must be resampled as a one (with probability
$p$).
Similarly, we have
\begin{eqnarray*}
Q_{1,1}(X)&=&\frac{(n-W)(n-W-1)}{n^2}p^2,
\\
Q_{-1,-1}(X)&=&\frac{W(W-1)}{n^2}(1-p)^2,
\end{eqnarray*}
since in order for the number of ones to increase by one from $X$ and then
again from $X'$,
at both steps a zero must be selected (with probability
$((n-W)/n)((n-W-1)/n)$) and then
at both steps the selected coordinate must be resampled as a one (with
probability $p^2$).
Now, basic properties of the binomial distribution show
\begin{eqnarray*}
q_1&=&\IE Q_{1}(X) = p(1-p),
\\
\Var\bigl(Q_1(X)\bigr)&=&\frac{p^3(1-p)}{n},\qquad  \Var(Q_{-1}(X)=
\frac{(1-p)^3 p}{n},
\\
\IE\bigl\llvert Q_{1,1}(X)-Q_1(X)^2\bigr
\rrvert &=&\frac{p^2}{n^2}\IE (n-W)= \frac
{p^2(1-p)}{n},
\\
\IE\bigl\llvert Q_{-1,-1}(X)-Q_{-1}(X)^2\bigr
\rrvert &=&\frac
{(1-p)^2}{n^2}\IE W= \frac{(1-p)^2
p}{n},
\end{eqnarray*}
and the result follows after putting these values into Theorem~\ref
{thm4} and
simplifying.
\end{pf}

\section{Applications}\label{sec3}

Because we are going to work in the total variation and local limit
metrics, we need to use a discrete analog of the normal distribution.
We use the \emph{translated Poisson} distribution, but any distribution
such that an analog of Lemma~\ref{lem7} below holds would also work
in the examples below (for example,
any standard discretization of the normal
distribution). We say that the random variable $Z$ has the translated Poisson
distribution,
denoted $Z\sim\TP(\mu, \sigma^2)$, if $Z-{\lfloor\mu-\sigma
^2\rfloor}
\sim
\Po(\sigma^2+\gamma)$, where $\gamma= \mu- \sigma^2 - {\lfloor
\mu
-\sigma^2\rfloor}$.
Note that $\IE Z = \mu$ and $\sigma^2\leq\Var Z \leq\sigma^2+1$.
The translated Poisson distribution is a Poisson distribution shifted by
an integer to closely match a given mean and variance; see R{\"o}llin \cite{Rollin2007a}
for basic properties and applications.

The following
lemma essentially states that we can use the translated Poisson
distribution as a discrete substitute for the normal distribution and also
provides bounds on the appropriate smoothing terms.

%
\begin{lemma}\label{lem7}
If $\mu\in\IR$ and $\sigma^2>0$, then as $\sigma\to\infty$,
%
\begin{eqnarray}
\label
{12}D_{k,m} \bigl(\TP\bigl(\mu,\sigma^2\bigr) \bigr)&=&\bigo
\bigl(\sigma^{-k}\bigr),
\\
\label{13}\dk \bigl(\TP\bigl(\mu,\sigma^2\bigr), \N\bigl(\mu,
\sigma^2\bigr) \bigr)&=&\bigo\bigl(\sigma^{-1}
\bigr),
\\
\label{14}\dw \bigl(\TP\bigl(\mu,\sigma^2\bigr), \N\bigl(\mu,
\sigma^2\bigr) \bigr)&=&\bigo (1)
\end{eqnarray}
and
%
\begin{equation}
\label{15} \sup_{k\in\IZ} \biggl\llvert \TP(\mu,\sigma)\{k\} -
\frac{1}{\sqrt{2\uppi\sigma^2}}\exp \biggl(-\frac{(k-\mu
)^2}{2\sigma^2} \biggr) \biggr\rrvert = \bigo
\bigl(\sigma^{-2} \bigr).
\end{equation}
\end{lemma}

%
\begin{remark}Let us make a few clarifying remarks about
Lemma~\ref{lem7} and its use in what follows. First note that as the
proof below shows, the rates
obtained in Lemma~\ref{lem7} hold in general for sums $X_1+\cdots
+X_n$ of
independent identically distributed random variables with integer
support and
$D_1(X_1)<2$. Also, in order to appreciate the Wasserstein bound \eqref
{14}, the reader should keep in mind that both distributions in the
statement are not standardized and that, for any random variables $X$
and $Y$ and any positive constant $c$,
%
\begin{equation}
\label{225}
\dw \bigl(\law(cX),\law(cY) \bigr) = c\dw \bigl(\law(X),\law (Y) \bigr).
\end{equation}
Hence, after scaling the variables in \eqref{14} by $\sigma^{-1}$, the
rate becomes the more familiar $\bigo(\sigma^{-1})$.
Finally, the statement in \eqref{15} is just the local limit theorem
for the
translated
Poisson distribution. Such a statement is only informative if the
right-hand side of \eqref{15} is $\lito(\sigma^{-1})$, because the
left-hand
side is trivially $\bigo(\sigma^{-1})$. In this section, we will prove
bounds for $\dloc (\law(W),\TP(\mu,\sigma^2) )$, which are
better than
$\bigo(\sigma^{-1})$ and therefore, by means of \eqref{15}, will lead
to a local
limit theorem for $W$ along with a bound on the rate of convergence.
\end{remark}

%
\begin{remark} \label{rem3} Lemma~\ref{lem7} can serve as a benchmark
for the best possible rates of convergence. For sums of i.i.d. random
variables under finite third moment conditions, the
Kolmogorov and Wasserstein distances between the normalized random
variables and the standard normal distribution
are both $\bigo(\sigma^{-1})$, which can be improved only under
additional assumptions (such as symmetry) of the involved
distributions. Furthermore, if the summands
are integer valued and smooth enough, then the local metric distance to
a discrete analog of the normal distribution has rate $\bigo(\sigma
^{-2})$ .

As indicated in the \hyperref[secintro]{Introduction}, our method will typically not yield
rates of convergence that are comparable to Lemma~\ref{lem7} and those
for sums of i.i.d. random variables.
So in applications where it is expected the rates should be the same as
those for sums of i.i.d. random variables
(e.g., magnetization in the Curie--Weiss model at high temperature),
our results are likely not optimal.
However, in particular for the local limit metric, for which only few
results with explicit rates of
convergence are known, it is not clear whether one can expect the same
rates as those for sums of i.i.d. variables,
and so we leave the question of optimality open.
\end{remark}

\begin{pf*}{Proof of Lemma~\ref{lem7}}
First, note that since $D_{n,m}(\law(X))=D_{n,m}(\law(X+l))$ for all integers
$l$, it is enough to prove \eqref{12} with the translated Poisson distribution
replaced by $\Po(\sigma^2+\gamma)$. We can represent this Poisson
distribution
as the convolution of $k$ independent Poisson distributions all having mean
$(\sigma^2+\gamma)/k$ and so by Lemmas \ref{lem3} and \ref{lem5} we
find that
for $X\sim\Po((\sigma^2+\gamma)/k)$,
%
\begin{equation}
D_{k,m} \bigl(\TP\bigl(\mu,\sigma^2\bigr) \bigr)\leq m
D_1\bigl(\law(X)\bigr)^k. \label{16}
\end{equation}
We can represent $X$ as the sum of ${\lfloor(\sigma^2+\gamma
)/k\rfloor}$
(here assume
$\sigma^2>k$) i.i.d. Poisson variables with means $\lambda_{\sigma
,k}\geq1$.
Lemma~\ref{lem6} now implies that if $Y_{\sigma,k}\sim\Po(\lambda
_{\sigma,k})$
and $D_1(\law(Y_{\sigma,k}))<2-\eps$ for some $\eps>0$ and all
$\sigma$ sufficiently
large, then
\begin{eqnarray*}
D_1\bigl(\law(X)\bigr) = \bigo\bigl(\sigma^{-1}\bigr),
\end{eqnarray*}
which with \eqref{16} yields \eqref{12}.
But it is well known (and easily checked) that for $W\sim\Po(\lambda)$
and any bounded function $g$, $\lambda\IE g(W+1)=\IE \{
Wg(W) \} $
and so
\begin{eqnarray*}
D_1\bigl(\law(W)\bigr)&=&\sup_{\llVert  g\rrVert _\infty\leq1}\bigl\llvert
\IE g(W+1)-\IE g(W)\bigr\rrvert
\\
&=&\frac{1}{\lambda} \sup_{\llVert  g\rrVert _\infty\leq1}\bigl\llvert \IE \bigl\{(W-
\lambda) g(W) \bigr\} \bigr\rrvert \leq\frac{1}{\sqrt{\lambda}},
\end{eqnarray*}
where the last inequality follows by Cauchy--Schwarz (using Fourier
methods, Barbour, Holst and Janson \cite{Barbour1992}, Proposition A.2.7, in
fact show that
$D_1(\law(W)) \leq2/{\sqrt{2e\lambda}}$). Thus, it is indeed true that
$D_1(Y_{\sigma,k})\leq1 < 2-\eps$ and \eqref{12} is proved.

The remaining properties follow by representing $\Po(\sigma^2+\gamma
)$ as a
sum of ${\lfloor\sigma^2\rfloor}$ i.i.d. Poisson random variables
and using well
known theory about sums of independent random variables: \eqref{13} and
\eqref{15}
are respectively Theorem~4 on page 111 and Theorem~6 on page 197
of Petrov \cite{Petrov1975} and \eqref{14} is
Corollary~4.2 on page 68
of Chen, Goldstein and Shao \cite{Chen2011}.
\end{pf*}

\subsection{Magnetization in the Curie--Weiss model}\label{sec4}

Let $\beta>0$, $h\in\IR$ and for $\mysigma\in\{ -1,1 \}^n$ define
the Gibbs
measure
%
\begin{equation}
\IP( \mysigma) = Z^{-1} \exp \biggl\{ \frac{\beta}{n} \sum
_{i < j} \mysigma_i \mysigma_j + h
\sum_i \mysigma_i \biggr\} ,
\label{17}
\end{equation}
where $Z$ is the appropriate normalizing constant (we use the letter
``$\mysigma$'' instead of the more commonly used ``$\sigma$'' in order to
avoid confusion with the notation for variance).

This probability model is referred to as the Curie--Weiss model and a
quantity of
interest is the magnetization $W=\sum_{i}\mysigma_i$ of the system. The
book Ellis \cite{Ellis1985}
provides a good introduction
to these models.
We use our framework to show total variation and local limit theorems
(LLTs) with bounds on the rates of convergence;
these are
stated in Theorem~\ref{thm7} below.
We start by stating known limit and approximation results for the
Kolmogorov metric, which we will need for our approach.

%
\begin{theorem}[{(Ellis, Newman and Rosen \cite{Ellis1980}, Theorem~2.2)}]\label{thm5}
If $\mysigma$ has law given by \eqref{17} with $0<\beta<1$, and
$h\in
\IR$
and $W=\sum_{i}\mysigma_i$,
then there is a unique solution $m_0$
of
\begin{eqnarray*}
m = \tanh(\beta m + h)
\end{eqnarray*}
and as $n\to\infty$,
\begin{eqnarray*}
\dk \biggl(\law \biggl({\frac{W-n m_0}{n^{1/2}}} \biggr), \N \biggl(0,{
\frac{1-m_0^2}{1-\beta+\beta m_0^2}} \biggr) \biggr)\to0.
\end{eqnarray*}
\end{theorem}

%
\begin{theorem}[({Eichelsbacher and L{\" o}we \cite{Eichelsbacher2010}, Theorems 3.3 and
3.7})]\label{thm6}
If $\mysigma$ has law given by \eqref{17} with $0<\beta<1$, and $h=0$
and $W=\sum_{i}\mysigma_i$,
then there is a constant $C$ depending only on $\beta$ such that
\begin{eqnarray*}
\dk \bigl(\law\bigl(n^{-1/2}W\bigr), \N\bigl(0,(1-\beta)^{-1}
\bigr) \bigr)\leq C n^{-1/2};
\end{eqnarray*}
the same bound holds for the Wasserstein metric.
\end{theorem}

Note that Chen, Fang and Shao \cite{Chen2012} have obtained
moderate deviation results,
which are much sharper than the Berry--Esseen type bounds of Theorem~\ref{thm6}.

The other ingredient of applying our framework here is to use Theorem~\ref{thm4}
to bound the necessary smoothing terms.
For this purpose,
let ${\mysigma}$ as
above and ${\mysigma}'$ be a step from ${\mysigma}$ in the
following reversible Markov chain: at each step of the chain a site
from the $n$ possible sites is chosen uniformly at random and then the
spin at that site is resampled according to the Gibbs measure \eqref{17}
conditional on the value of the spins
at all other sites. Let $W=\sum_{i=1}^n \mysigma_i$ and
$W'=\sum_{i=1}^n\mysigma_i'$
and note that $(W,W')$ is an exchangeable pair. Finally,
define
\begin{eqnarray*}
Q_{m} = \IP\bigl[W' = W+m|{\mysigma}\bigr],
\end{eqnarray*}
$q_m=\IE Q_m$, and
\begin{eqnarray*}
Q_{m_1,m_2}= \IP\bigl[W' = W+m_1,W''
= W'+m_2|{\mysigma}\bigr],
\end{eqnarray*}
where $W''$ is obtained from $W'$ in the same way that $W'$ is obtained from
$W$
(i.e., $(W,W',W'')$ are the magnetizations in three consecutive steps
in the
stationary
Markov chain described above).
We have the following result, proved at the end of this section.

%
\begin{lemma}\label{lem8}
If $0<\beta<1$, $h\in\IR$ and $M=\frac{1}{n}\sum_{i=1}^n \mysigma
_i$, then there
is a unique
solution
$m_0$ to
\begin{eqnarray*}
m = \tanh(\beta m + h),
\end{eqnarray*}
and for $k=\pm2$,
%
\begin{eqnarray}
\label{19} \biggl\llvert Q_k-\frac{1-m_0^2}{4}\biggr\rrvert
&\leq& C \biggl(\llvert M-m_0\rrvert +\frac{1}{n} \biggr),
\\
\label{20} \bigl\llvert Q_{k,k}-Q_k^2\bigr
\rrvert &=& \bigo\bigl(n^{-1}\bigr),
%
\\
\label{21} \biggl\llvert q_k-\frac{1-m_0^2}{4}\biggr\rrvert &=&
\bigo\bigl(n^{-1/2}\bigr), \qquad \Var(Q_k)=\bigo
\bigl(n^{-1}\bigr)
\end{eqnarray}
and
%
\begin{equation}
\label{22} D_{2,2}(W)=\bigo\bigl(n^{-1}\bigr).
\end{equation}
\end{lemma}

We can now put these pieces together to
obtain total variation and local limit convergence theorems with bounds
on the rates for
the magnetization.

%
\begin{theorem}\label{thm7} Let
$\mysigma$ have law given by \eqref{17}, $W=\sum_{i}\mysigma_i$,
and let $\delta=\delta(n)=(1-(-1)^n)/2$.
For $0<\beta<1$ and $h=0$, there is a constant $C$ that depends
only
on $\beta$ such that
\begin{eqnarray*}
\dloc \biggl(\law \bigl((W+\delta)/2 \bigr),\TP \biggl(0,{\frac
{n}{4(1-\beta)}}
\biggr) \biggr)&\leq& C n^{-3/4},
\\
\dtv \biggl(\law \bigl((W+\delta)/2 \bigr),\TP \biggl(0,{\frac
{n}{4(1-\beta)}}
\biggr) \biggr)&\leq& C n^{-1/3}.
\end{eqnarray*}
If $0<\beta<1$, $h\in\IR$, and $m_0$ is as in Theorem~\ref{thm5}, then
\begin{eqnarray*}
\dloc \biggl(\law \biggl({\frac{W+\delta}{2}} \biggr), \TP \biggl({
\frac{m_0}{2}},{\frac{n(1-m_0^2)}{4(1-\beta+\beta
m_0^2)}} \biggr) \biggr) = \lito
\bigl(n^{-1/2} \bigr)
\end{eqnarray*}
as $n\toinf$.
\end{theorem}

\begin{pf}
The theorem follows from (ii), (iii) and (iv) of Theorem~\ref{thm2}
with $m=1$ and $l=2$,
Lemma~\ref{lem7}, Theorems \ref{thm6} and \ref{thm5}, and the bounds
on the
smoothing terms in Lemma~\ref{lem8}.
\end{pf}

\begin{remark}
In the critical case where $\beta=1$, optimal bounds on the Kolmogorov
and Wasserstein
distances between the magnetization (appropriately normalized) and its
non-normal limiting distribution have been obtained by Eichelsbacher and L{\" o}we
\cite{Eichelsbacher2010}, Theorem~3.3, and Chatterjee and Shao \cite{Chatterjee2011a}, Theorem~2.1. In
fact, the smoothing bounds of Lemma~\ref{lem8} can be shown to apply
to this
case with $h=0$ and an appropriate analog of Lemma~\ref{lem7} also
holds for a
discretization of the non-normal limiting distribution. And these two facts
can be used to prove new bounds on the total variation distance between the
magnetization and a discrete version of this limiting distribution (although
after working out the details, we are not able to obtain meaningful
local metric results). However, we omit this result due to the inappropriate
amount of space it would take for a precise formulation.
\end{remark}

\begin{pf*}{Proof of Lemma~\ref{lem8}}
We only consider $k=2$, the case
$k=-2$ being similar.
An easy calculation shows that
\begin{eqnarray*}
\IP\bigl(s'_i = 1 | (s_j)_{j \neq i}
\bigr) = \frac{ \exp
\{ \sklfrac{\beta}{n} \sum_{j \neq i} s_j + h  \} } {
\exp \{ \sklfrac{\beta}{n} \sum_{j \neq i} s_j + h  \}
+ \exp \{ -\sklfrac{\beta}{n} \sum_{j \neq i} s_j - h  \}
}.
\end{eqnarray*}
Denoting $m_i:=n^{-1}\sum_{j\neq i}\mysigma_i$, we have
\begin{eqnarray*}
Q_2=\frac{1}{n}\sum_{i=1}^n
\frac{1-\mysigma_i}{2} \IP\bigl(s'_i = 1 |
(s_j)_{j \neq i} \bigr) =\frac{1}{n}\sum
_{i=1}^n\frac{1-\mysigma_i}{2} \frac{\tanh(\beta m_i+ h)+1}{2},
\end{eqnarray*}
since in order for the Markov chain to increase by two, a site in state ``$-1$''
must be selected and then changed to ``$+1$''.
Now, some simplification shows
\begin{eqnarray*}
Q_2 &=&\frac{1}{4}-\frac{M}{4}+\frac{\tanh(\beta M + h)}{4}(1-M)
\\
&&{} + \frac{1}{4n}\sum_{i=1}^n(1-
\mysigma_i) \bigl(\tanh(\beta m_i +h)-\tanh(\beta M + h)
\bigr).
\end{eqnarray*}
Thus, we find
\begin{eqnarray*}
\biggl\llvert Q_2-\frac{1-m_0^2}{4}\biggr\rrvert &\leq&
\frac{1}{4}\llvert M-m_0\rrvert +\frac{1}{4}\bigl\llvert
\tanh(\beta m_0+h)-\tanh (\beta M + h)\bigr\rrvert
\\
&&{} +\frac{1}{4}\bigl\llvert M\tanh(\beta M + h)-m_0\tanh(\beta
m_0+h)\bigr\rrvert
\\
&&{} +\frac{1}{4n}\sum_{i=1}^n(1-
\mysigma_i)\bigl\llvert \tanh(\beta m_i+h)-\tanh(\beta
M+h)\bigr\rrvert .
\end{eqnarray*}
Since $\tanh(x)\in(-1,1)$ is $1$-Lipschitz and $-1\leq M\leq1$ the
first part
of the
claim now easily follows.

For the second assertion, note that
\begin{eqnarray*}
Q_{2,2} & =& \frac{1}{16 n^2}\sum_{i\neq j}
(1-\mysigma_i) \bigl(\tanh(\beta m_i+h)+1\bigr)
\\
&&\hphantom{\frac{1}{16 n^2}\sum_{i\neq j}}{}\times (1-\mysigma_j) \bigl(\tanh\bigl(\beta
m_{i,j}+\beta n^{-1}+h\bigr)+1\bigr),
\end{eqnarray*}
where $m_{i,j}=\beta n^{-1}\sum_{k\neq i,j}\mysigma_k$, and also that
\begin{eqnarray*}
Q_2^2  = \frac{1}{16n^2}\sum
_{i,j}(1-\mysigma_i) \bigl(\tanh(\beta
m_i+h)+1\bigr) (1-\mysigma_j) \bigl(\tanh(\beta
m_j+h)+1\bigr).
\end{eqnarray*}
We can now find
\begin{eqnarray*}
\bigl\llvert Q_{2,2}-Q_2^2\bigr\rrvert & \leq&
\frac{1}{8n^2}\sum_{i}(1-\mysigma_i)
\bigl(\tanh(\beta m_i + h)+1\bigr)^2
\\
&&{} + \frac{1}{8n^2}\sum_{i\neq j}(1-
\mysigma_i) (1-\mysigma_j) \bigl\llvert \tanh(\beta
m_i + h)+1\bigr\rrvert
\\
&&\hphantom{{} + \frac{1}{8n^2}\sum_{i\neq j}}{}\times\bigl\llvert \tanh\bigl(\beta m_{i,j}+\beta
n^{-1} + h\bigr)-\tanh (\beta m_j + h)\bigr\rrvert .
\end{eqnarray*}
Straightforward estimates now yield \eqref{20}.

The assertions of \eqref{21} follow from
\eqref{19} and the fact $\IE\llvert  M-m_0\rrvert ^j =
\bigo (n^{-j/2} )$ which is obtained
from standard concentration results; see, for example, Chatterjee
\cite{Chatterjee2007}, Proposition~1.3. Finally, \eqref{22} follows from \eqref{19},
\eqref
{20}, and \eqref{21}
applied to Theorem~\ref{thm4}.
\end{pf*}

\subsection{Isolated vertices in the \texorpdfstring{Erd\H{o}s--R\'enyi
}{Erdos-Renyi} random
graph}\label{sec5}

In this and the next section, we will derive LLTs for the number of isolated
vertices and triangles in the Erd\H{o}s--R\'enyi random graph. There do
not appear to be
many results showing LLTs for random graph variables (and even fewer
having error bounds) although one area that has seen activity is
showing LLTs
for the size of the maximal component in graphs and hypergraphs; see
Stepanov \cite{Stepanov1970}, Karo{\'n}ski and {\L }uczak \cite{Karonski2002} and Behrisch, Coja-Oghlan and Kang
\cite{Behrisch2010a}. An
alternative approach to proving an LLT for the number of isolated
vertices in
an Erd\H{o}s--R\'enyi graph which we do not believe has been pursued
would be to use the
results of Bender, Canfield and McKay \cite{Bender1997} which have
detailed formulas for the
number of
graphs with a given number of vertices and edges and no isolated
vertices. To
the best of our knowledge, the following results on the number of isolated
vertices and number of triangles are new.

Before proceeding, we make a remark to prepare the dedicated reader for
the proofs below.
Many proofs of limit theorems for random graph variables
involve tedious moment calculations. For example, the limit results we
use below in our framework: Ruci{\'n}ski \cite
{Rucinski1988} uses the method of moments
to derive conditions where the number of copies of a ``small'' subgraph
in an Erd\H{o}s--R\'enyi graph will be approximately normally
distributed, and Barbour, Karo{\'n}ski and Ruci{\' n}ski \cite{Barbour1989}
uses a variation of Stein's method which in turn relies on moment estimates
to show limit theorems for the number of copies of certain subgraphs
in an Erd\H{o}s--R\'enyi graph; see also the references in these
documents. Since our
theory relies on
bounding means and variances of conditional probabilities, our work
below continues this tradition.

We use our framework to show total variation and local limit theorems
with bounds on the rates
for the number of isolated vertices; this is
Theorem~\ref{thm9} below.
We start by stating known limit and approximation results.
Define $G=G(n,p)$ to be a
random graph with $n$ vertices where each edge appears with probability $p$,
independent of all other edges.

%
\begin{theorem}[(Barbour, Karo{\'n}ski and Ruci{\' n}ski \cite{Barbour1989}, Kordecki \cite{Kordecki1990})]\label{thm8}
Let $W=W(n,p)$ be the number of isolated vertices of $G(n,p)$, and let
$\tilde{W}$ be $W$ normalized to have zero mean and unit variance.
Then $\ww{W}$
converges in distribution to the standard normal if and only if
%
\begin{equation}
\label{23} \lim_{n\to\infty}n^2p=\infty\quad \mbox{and}\quad
\lim_{n\to\infty}\bigl(\log (n)-np\bigr)=\infty.
\end{equation}
In that case, with $\sigma^2_n=\Var W$,
\begin{eqnarray*}
\dk\bigl(\law(\tilde{W}),\Phi\bigr) = \bigo\bigl(\sigma_n^{-1}
\bigr).
\end{eqnarray*}
\end{theorem}

The conditions of convergence was proved by Barbour, Karo{\'n}ski and Ruci{\' n}ski
\cite{Barbour1989},
whereas the
bounds for the Kolmogorov metric was obtained by Kordecki \cite{Kordecki1990}.

The other ingredient of applying our framework here is to use Theorem~\ref{thm4}
to bound the necessary smoothing terms.
We have the following result, proved at the end of this section.

%
\begin{lemma}\label{lem9}
Let $W=W(n,p)$ be the number of isolated vertices in an Erd\H{o}s--R\'
enyi graph $G(n,p)$
and $\sigma_n^2=\Var W$.
\begin{enumerate}[(ii)]
\item[(i)] If $\lim_{n\to\infty}(\log(n)-np)=\infty$, and either
$\lim_{n\to\infty}np=\infty$ or $\lim_{n\to\infty}np=c>0$,
then $\sigma_n^2\asymp n\mathrm{e}^{-np}$ and
\begin{eqnarray*}
D_1(W)= \bigo \bigl(\sigma_n^{-1} \bigr),\qquad
D_2(W)= \bigo \bigl(\sigma_n^{-2} \bigr).
\end{eqnarray*}
\item[(ii)] If $\lim_{n\to\infty}np=0$ and
$\lim_{n\to\infty}n^2p=\infty$, then
$\sigma_n^2\asymp n^2 p$ and
\begin{eqnarray*}
\begin{array}{rcl@{\qquad}rcl}
D_1(W)&=& \bigo \bigl( (\sqrt{np}\sigma_n
)^{-1} \bigr), &D_2(W)&=& \bigo \bigl((np)^{-1}
\sigma_n^{-2} \bigr), \vspace{3pt}
\\
D_{1,2}(W)&=&\bigo \bigl(\sigma_n^{-1} \bigr),&
D_{2,2}(W)&=& \bigo \bigl(\sigma_n^{-2} \bigr).
\end{array}
\end{eqnarray*}
\end{enumerate}
\end{lemma}

We now summarize the results of our framework combined with
Theorem~\ref{thm8}
and Lemma~\ref{lem9}. For two distribution functions $F$ and
$G$ with integer
support, let
\begin{eqnarray*}
\dloc^{m}(F,G) = \bigl\llVert \D \bar{F}^m - \D \bar{G}^m\bigr\rrVert
_\infty.
\end{eqnarray*}
Note that $\dloc= \dloc^{1}$ and recall also the equality given by
Lemma~\ref{lem1}.

%
\begin{theorem}\label{thm9}
Let $W=W(n,p)$ be the number of isolated vertices in an Erd\H{o}s--R\'
enyi random graph
$G(n,p)$ and $\tilde{W}$ be $W$ normalized to have zero mean and unit
variance.
With $\mu_n=\IE W$ and $\sigma_n^2=\Var(W)$, we have the following.
\begin{enumerate}[(ii)]
\item[(i)] If $\lim_{n\to\infty}(\log(n)-np)=\infty$, and either
$\lim_{n\to\infty}np=\infty$ or $\lim_{n\to\infty}np=c>0$,
%
\begin{eqnarray*}
\dloc \bigl(\law(W),\TP\bigl(\mu_n,\sigma_n^2
\bigr) \bigr) = \bigo \bigl(\sigma_n^{-3/2} \bigr).
\end{eqnarray*}
\item[(ii)] If $\lim_{n\to\infty}np=0$ and
$\lim_{n\to\infty}n^2p=\infty$, then
%
\begin{eqnarray*}
\dloc \bigl(\law(W),\TP\bigl(\mu_n,\sigma_n^2
\bigr) \bigr) &=&\bigo \bigl(\sigma_n^{-1}
\bigl(np^{3/4} \bigr)^{-1} \bigr),
\\
\dloc^{2} \bigl(\law(W),\TP\bigl(\mu_n,\sigma_n^2
\bigr) \bigr) &=&\bigo \bigl(\sigma_n^{-3/2} \bigr).
\end{eqnarray*}
\end{enumerate}
\end{theorem}

\begin{pf}
The result follows from (ii) of Theorem~\ref{thm2} with $l=2$, using the
known
rates stated above
in Theorem~\ref{thm8} coupled with Lemma~\ref{lem7}, and bounds on
the smoothing
quantities provided by
Lemma~\ref{lem9}.
\end{pf}

\begin{remark}\label{rem1}
The bounds on the smoothing quantities presented below can be written
in terms
of $n$ and $p$, so that the asymptotic results of the theorem can be written
explicitly whenever the bounds on $\dk(\tilde{W},\Phi)$ are also explicit.
\end{remark}

%
\begin{remark}\label{rem2}
The second case of the theorem is interesting and deserves elaboration.
In some
regimes, our bounds do not imply a LLT in the natural lattice of span
one (e.g.,
$p\asymp n^{-\alpha}$, where $\alpha>4/3$), but we can obtain a
useful bound
on the rate
of convergence in
the $\dloc^{2}$ distance. This implies that the approximation is better
by averaging the probability mass function of $W$ over two neighboring
integers
and then comparing this value to its analog for the normal density. One
explanation for this phenomenon is that in such a regime, the graph $G(n,p)$
will be extremely sparse so that parity of $W$ will be dominated by the number
of isolated edges. In other words, with some
significant probability, $n-W$ will be approximately equal to twice the number
of isolated edges, in
which case
we would not expect the normal density to be a good approximation for
each point
on the integer lattice.
\end{remark}

\begin{pf*}{Proof of Lemma~\ref{lem9}}
The stated order of the variance follows from
\begin{eqnarray*}
\sigma_n^2=n(1-p)^{n-1} \bigl[1+(np-1)
(1-p)^{n-2} \bigr],
\end{eqnarray*}
which follows easily after representing $W$ as a sum of indicators;
see also the moment information given below.

To prove the bounds on the smoothing terms, we apply Theorems \ref
{thm3} and \ref{thm4}.
For this purpose, let $G(n,p)$ as above and $G'(n,p)$ be a step
from $G(n,p)$ in the following reversible Markov chain:
from a given graph $G$ the chain moves to $G'$ by
choosing
two vertices uniformly
at random and resampling the ``edge'' between them.
Let $W=W(n,p)$ be the number of isolated vertices of the Erd\H{o}s--R\'
enyi graph
$G=G(n,p)$ and  $W'$ be the number of isolated vertices after one
step in
the chain
from $G$, so $(W,W')$ is an exchangeable pair.
Finally, define
\begin{eqnarray*}
Q_{m} = \IP\bigl[W' = W+m|{\sigma}\bigr],
\end{eqnarray*}
$q_m=\IE Q_m$, and
\begin{eqnarray*}
Q_{m_1,m_2}= \IP\bigl[W' = W+m_1,W''
= W'+m_2|{\sigma}\bigr],
\end{eqnarray*}
where $W''$ is obtained from $W'$ in the same way that $W'$ is obtained from
$W$
(i.e. $(W,W',W'')$ are the magnetizations in three consecutive steps in the
stationary
Markov chain described above).

In order to compute the terms needed to apply
Theorems \ref{thm3} and \ref{thm4}, we need some auxiliary random variables.
Let
$W_k$ be the number of vertices of degree $k$ in $G$ (so $W_0\equiv W$),
and $E_2$ be the number of connected pairs of vertices each having
degree one
(i.e., $E_2$ is the number of isolated edges).
We have
\begin{eqnarray*}
\begin{array}{rcl@{\qquad}rcl}
Q_1(G)&=&\displaystyle \frac{(W_1-2E_2)}{\binom{n}{2}}(1-p), &Q_{-1}(G)&=&
\displaystyle \frac{W(n-W)}{\binom{n}{2}}p,
\\\noalign{\vspace*{2pt}}
Q_{1,1}(G)&=&\displaystyle \frac{2\binom{W_1-2E_{2}}{2}}{\binom{n}{2}^2}(1-p)^2,&
Q_{-1,-1}(G)&=&\displaystyle \frac{4\binom{W}{2}\binom{n-W+1}{2}}{\binom
{n}{2}^2}p^2,
\\\noalign{\vspace*{2pt}}
Q_2(G)&=&\displaystyle \frac{E_2}{\binom{n}{2}}(1-p),& Q_{-2}(G)&=&
\displaystyle \frac{\binom{W}{2}}{\binom{n}{2}} p,
\\\noalign{\vspace*{2pt}}
Q_{2,2}(G)&=&\displaystyle \frac{2\binom{E_2}{2}}{\binom{n}{2}^2}(1-p)^2,&
Q_{-2,-2}(G)&=&\displaystyle \frac{\binom{W}{2}\binom{W-2}{2}}{\binom{n}{2}^2}p^2.
\end{array}
\end{eqnarray*}
These equalities are obtained through straightforward considerations. For
example, in order for one step in the chain to increase the number of isolated
vertices by one, an edge of $G$ must be chosen that has exactly one end vertex
of degree one, and then must be removed upon resampling. For one step
in the
chain to decrease the number of isolated vertices by one, an isolated vertex
must be connected to a vertex with positive degree.

From this point, the lemma will follow after computing the pertinent moment
information needed to apply Theorems \ref{thm3} and \ref{thm4}. By
considering
appropriate indicator functions, it is an elementary combinatorial
exercise to
obtain
\begin{eqnarray*}
\IE W_1&=&2\binom{n} {2}p(1-p)^{n-2},\qquad  \IE E_2=
\binom{n} {2}p(1-p)^{2n-4},
\\
\IE W_1^2&=&2\binom{n} {2}p(1-p)^{n-2}+2p
\binom{n} {2} \bigl[(1-p)^{2n-4}+p(n-2)^2(1-p)^{ 2n-5}
\bigr],
\\
\IE E_2^2&=&\binom{n} {2}p(1-p)^{2n-4}+6\binom{n}
{4}p^2(1-p)^{4n-12},
\\
\IE W_1E_2&=&\binom{n} {2}p(1-p)^{2n-4}
\bigl[(n-2) (n-3)p(1-p)^{n-4}+2 \bigr],
\end{eqnarray*}
which will yield the results for negative jumps, and
\begin{eqnarray*}
\IE W&=&n(1-p)^{n-1},\qquad  \IE W^2=n(1-p)^{n-1}+2
\binom{n} {2}(1-p)^{2n-3},
\\
\IE W^3&=&n(1-p)^{n-1}+6\binom{n} {2}(1-p)^{2n-3}+6
\binom {n} {3}(1-p)^{3n-6},
\\
\IE W^4&=&n(1-p)^{n-1}+14\binom{n} {2}(1-p)^{2n-3}
\\
&&{} +36\binom{n} {3}(1-p)^{3n-6}+24\binom{n} {4}(1-p)^{4n-10},
\end{eqnarray*}
which will yield the results for the positive jumps. Theorems \ref{thm3}
and \ref{thm4} now give the desired rates. As an example of these
calculations,
note that
%
\begin{equation}\label{24}
\frac{\Var Q_1(G)}{q_1^2} =\frac{\IE W_1^2-4\IE W_1E_2+4\IE
E_2^2-(\IE
W_1-2\IE E_2)^2}{(\IE W_1-2\IE E_2)^2},
\end{equation}
which after the dust settles is $\bigo (n^{-1}\mathrm{e}^{np} )$ in
case (i)
of the theorem. Similarly, since $q_{1}=\IE Q_{-1}$, we have
%
\begin{equation}\label{25}
\frac{\Var Q_{-1}(G)}{q_{1}^2}=\frac{\IE W^4 -2n\IE W^3 + n^2\IE W^2
-(n
\IE W- \IE W^2)^2}{(n \IE W- \IE W^2)^2},
\end{equation}
which is again $\bigo (n^{-1}\mathrm{e}^{np} )$ in case (i).
Theorem~\ref{thm3} implies that $D_1(W)$ is bounded above
by the sum of the square roots of the terms in \eqref{24} and \eqref
{25} so
that in case (i),
\begin{eqnarray*}
D_1(W) = \bigo \bigl(n^{-1/2}\mathrm{e}^{\vfrac{np}{2}} \bigr) = \bigo
\bigl(\sigma _n^{-1} \bigr).
\end{eqnarray*}
For the second part of (i), note that
\begin{eqnarray*}
\frac{\IE\llvert  Q_{1,1}-Q_1^2\rrvert }{q_1^2}\leq\frac
{\IE W_1+2\IE
E_2}{(\IE
W_1-2\IE E_2)^2},
\end{eqnarray*}
which is $\bigo (n^{-2}p^{-1}\mathrm{e}^{np} )$ in case (i) and
\begin{eqnarray*}
\frac{\IE\llvert  Q_{-1,-1}-Q_{-1}^2\rrvert }{q_1^2}=\frac
{\IE
(W(n-W)\llvert  n-2W+1\rrvert  )}{(n \IE W- \IE
W^2)^2}\leq\frac{n+1}{ n \IE W- \IE W^2},
\end{eqnarray*}
which is $\bigo (n^{-1}\mathrm{e}^{np} )$ in case (i),
so that we have
\begin{eqnarray*}
D_2(W) = \bigo \bigl(n^{-1}\mathrm{e}^{np} \bigr) = \bigo
\bigl(\sigma _n^{-2} \bigr).
\end{eqnarray*}
This proves (i); the remaining bounds are similar and
omitted for the sake of brevity.
\end{pf*}

\subsection{Triangles in the \texorpdfstring{Erd\H{o}s--R\'
enyi}{Erdos-Renyi} random
graph}\label{sec6}

In this section, we use our framework to
first obtain a new bound on the rate of convergence in the total
variation distance between
the normal distribution and
the number of triangles in an Erd\H{o}s--R\'enyi random graph.
We then use this
new rate to obtain a local
limit theorem for this example. As in
Section~\ref{sec5}, define $G=G(n,p)$ to be a random graph with $n$
vertices
where each edge appears with probability $p$, independent of all other edges.
From this point, we have the following theorem.

%
\begin{theorem}[(Ruci{\'n}ski \cite{Rucinski1988},
Barbour, Karo{\'n}ski and Ruci{\' n}ski \cite{Barbour1989})]\label{thm10}
Let $W=W(n,p)$ be the number of triangles of $G(n,p)$, and let $\tilde
{W}$ be
$W$ normalized to have zero mean and unit variance. Then $\ww{W}$
converges to
the
standard normal if and only if
\begin{eqnarray*}
\lim_{n\to\infty}np=\infty\quad \mbox{and} \quad \lim_{n\to\infty
}n^2(1-p)=
\infty.
\end{eqnarray*}
In that case, with $\sigma^2_n = \Var W$,
\begin{eqnarray*}
\dw\bigl(\law(\tilde{W}),\Phi\bigr) = \bigo\bigl(\sigma_n^{-1}
\bigr).
\end{eqnarray*}
\end{theorem}

The other ingredient of applying our framework here is to use Theorem~\ref{thm4}
to bound the necessary smoothing terms.
We have the following result, proved at the end of this section.

%
\begin{lemma}\label{lem10}
Let $W=W(n,p)$ be the number of triangles in an Erd\H{o}s--R\'
enyi random graph
$G(n,p)$. If
$n^{\alpha}p\to c>0$ with $1/2\leq\alpha<1$ then $\Var(W)\asymp n^3
p^3$ and
%
\begin{equation}
\label{26} D_1(W)= \bigo \bigl(\sigma_n^{-1}
\bigr),\qquad  D_{2}(W)=\bigo \bigl(\sigma_n^{-2}
\bigr).
\end{equation}
\end{lemma}

We now summarize the results derived from the bound of Theorem~\ref{thm10}
coupled with our theory above.

%
\begin{theorem}\label{thm11}
Let $W=W(n,p)$ be the number of triangles in an Erd\H{o}s--R\'
enyi random graph
$G(n,p)$. If
$n^{\alpha}p\to c>0$ with $1/2\leq\alpha<1$ then with
$\mu_n=\IE W$ and $\sigma_n^2:=\Var(W)$, we have
\begin{eqnarray*}
\dtv \bigl(\law(W),\TP\bigl(\mu_n,\sigma^2_n
\bigr) \bigr) =\bigo \bigl(n^{-(1-\alpha)} \bigr)
\end{eqnarray*}
and
\begin{eqnarray*}
\dloc \bigl(\law(W),\TP\bigl(\mu_n,\sigma^2_n
\bigr) \bigr) =\bigo \bigl(\sigma_n^{-1}n^{-(1-\alpha)/2}
\bigr).
\end{eqnarray*}
\end{theorem}

\begin{pf}
The result follows from (iv) and then (i) (or (iii)) of
Theorem~\ref{thm2}
with $l=2$ and $m=1$,
using the known rates stated above in
Theorem~\ref{thm10} coupled with Lemma~\ref{lem7} and bounds on the smoothing
quantities provided by
Lemma~\ref{lem10}.
\end{pf}

\begin{remark} It is worthwhile noting that we obtain the LLT only for
those values of $\alpha$ for which we have $\IE W \asymp\Var W$. In contrast,
if $0<\alpha<1/2$, we have that $\IE W\asymp n^{3-3\alpha}$, whereas
$\Var W
\asymp
n^{4-5\alpha} \gg\IE W$. It is not clear if this is an artifact of
our method or if a
standard LLT
does not hold in this regime;
cf. Remark~\ref{rem2} following Theorem~\ref{thm9}.
\end{remark}

In order to prove Lemma~\ref{lem10}, we will apply Theorems \ref{thm3}
and \ref{thm4} by constructing a Markov chain on graphs with $n$ vertices
which
is reversible with respect to the law of $G(n,p)$. From a given graph $G$,
define a step in the chain to $G'$ by choosing two vertices of $G$ uniformly
at
random and independently resampling the ``edge'' between them. It is
clear that
this Markov chain is reversible with respect to the distribution of $G(n,p)$.
We are now in a position to compute the terms needed to apply
Theorems \ref{thm3} and \ref{thm4}.

%
\begin{lemma}\label{lem11}
Let $(W,W')$ be the number of triangles in the exchangeable pair of
Erd\H{o}s--R\'enyi graphs
$(G,G')$ as defined above. If $Q_1(G)=\IP[W'=W+1|G]$, then
%
\begin{eqnarray}
&&\Var Q_1(G)
\nonumber
\\
\label{27}&&\quad  \leq\frac{(n-2)}{\binom{n}{2}}p^4(1-p) \bigl(1-p^2
\bigr)^{n-3} \bigl(1-p^2(1-p) \bigl(1-p^2
\bigr)^{n-3} \bigr)
\\
\label{28}&&\qquad {} +\frac{4\binom{n-2}{2}}{\binom{n}{2}}p^5(1-p)^2 \bigl(
\bigl(1-2p^2+p^3\bigr)^{n-4}-p
\bigl(1-p^2\bigr)^{2n-6} \bigr)
\\
\label{29}&&\qquad {} +\frac{4\binom{n-2}{2}}{\binom{n}{2}}p^5(1-p)^2\bigl(1-p^2
\bigr)^{2n-8} \bigl(1-p-p\bigl(1-p^2\bigr)^{2}
\bigr)
\\
\label{30}&&\qquad {} +\frac{12\binom{n-2}{3}}{\binom{n}{2}}p^6(1-p)^2 \bigl((1-p)^{n-3}
\bigl(1+p-p^2\bigr)^{n-5}-\bigl(1-p^2
\bigr)^{2n-6} \bigr)
\\
\label{31}&&\qquad {} +\frac{12\binom{n-2}{3}}{\binom{n}{2}}p^6(1-p)^2\bigl(1-p^2
\bigr)^{2n-9} \bigl(-2p+4p^2-3p^4+p^6
\bigr)
\\
\label{32}&&\qquad {} +\frac{3\binom{n-2}{3}}{\binom{n}{2}}p^6(1-p)^2\bigl(1-p^2
\bigr)^{2n-10} \bigl(4p^3-7p^4+4p^6-p^8
\bigr)
\\
\label{33}&&\qquad {} +\frac{12\binom{n-2}{4}}{\binom{n}{2}}p^6(1-p)^2\bigl(1-p^2
\bigr)^{2n-10} \bigl(4p^3-7p^4+4p^6-p^8
\bigr).
\end{eqnarray}
\end{lemma}

\begin{pf}
Let $X_{i,j}$ be the indicator that there is an edge between vertices
$i$ and
$j$ and $V_i^{k,j}:=X_{i,j}X_{i,k}$ be the indicator that there is a $V$-star
on
the vertices $\{i, j, k\}$ with elbow $i$. We easily find
\begin{eqnarray*}
Q_1(G)=\frac{p}{\binom{n}{2}}\sum_{\{j,k\}}\sum
_{i\neq j,k}Y_{i}^{j,k},
\end{eqnarray*}
where we define the indicator variables
\begin{eqnarray*}
Y_{i}^{j,k}=(1-X_{j,k})V_i^{j,k}
\prod_{l\neq i,j,k}\bigl(1-V_l^{j,k}
\bigr).
\end{eqnarray*}
From this point, we note that the variance of $Q_1(G)$ is a sum of covariance
terms times $p^2/\binom{n}{2}^2$. For fixed $i, j$, and $k$, there are
$3\binom{n}{3}$ terms of the form $\Cov(Y_i^{j,k}, Y_u^{s,t})$, where
we are
including $\Var(Y_i^{j,k})$. In order to compute this sum, we will
group these
covariance terms with respect to the number of indices $Y_i^{j,k}$ and
$Y_u^{s,t}$ share, which will yield the lemma after computing their covariances.

As an example of the type of calculation involved in computing these covariance
terms, note that
\begin{eqnarray*}
\IE Y_i^{j,k}=p^2(1-p) \bigl(1-p^2
\bigr)^{2n-6},
\end{eqnarray*}
so that
\begin{eqnarray*}
\Var Y_i^{j,k} =p^2(1-p) \bigl(1-p^2
\bigr)^{2n-6}\bigl(1-p^2(1-p) \bigl(1-p^2
\bigr)^{2n-6}\bigr).
\end{eqnarray*}
Furthermore, for $j\neq s$, we find that
\begin{eqnarray*}
\IE \bigl\{Y_i^{j,k}Y_{i}^{s,k} \bigr
\} =p^3(1-p)^2\bigl((1-p)^3+3p(1-p)^2+p^2(1-p)
\bigr)^{n-4}.
\end{eqnarray*}
Below we focus on carefully spelling out the number and types of covariance
terms that contribute to the variance of $Q_1(G)$ and leave to the reader
detailed calculations similar to those above.

If $\{i,j,k\}=\{u,s,t\}$ and $Y_i^{j,k}\neq Y_u^{s,t}$, then
$\IE \{Y_i^{j,k}
Y_u^{s,t} \} =0$, so the corresponding covariance term is
negative, which
we bound
above by zero. In the case that $Y_i^{j,k}=Y_u^{s,t}$, we obtain a variance
term
which corresponds to \eqref{27} in our bound.

Assume now that $\{i,j,k\}$ and $\{u,s,t\}$ have exactly two elements in
common
and consider which indices are equal. In the cases that $Y_u^{s,t}$ is equal
to
$Y_u^{j,k}$, $Y_j^{k,t}$, $Y_u^{i,k}$, $Y_u^{i,j}$ or $Y_k^{j,t}$, then
$\IE
\{Y_i^{j,k} Y_u^{s,t} \} =0$, so the corresponding
covariance term is
negative,
which we bound above by zero. The two remaining cases to consider are
$Y_{u}^{s,t}=Y_i^{s,k}$ which contribute $2(n-3)$ equal covariance terms
leading
to \eqref{28}, and $Y_{u}^{s,t}=Y_j^{i,t}$ which also contribute
$2(n-3)$ equal
covariance terms leading to \eqref{29}.

Assume $\{i,j,k\}$ and $\{u,s,t\}$ have exactly one element in common;
we have
four cases to consider. There are $2(n-3)(n-4)$ covariance terms of the basic
form $Y_u^{s,t}=Y_u^{j,t}$ which leads to \eqref{30}, there are $2(n-3)(n-4)$
covariance terms of the basic forms $Y_u^{s,t}=Y_u^{i,t}$ or
$Y_u^{s,t}=Y_j^{s,t}$ which leads to \eqref{31}, and there are $\binom
{n-3}{2}$
terms of the form $Y_u^{s,t}=Y_i^{s,t}$ which yields \eqref{32}.

Finally, if $\{i,j,k\}$ and $\{u,s,t\}$ are distinct sets, of which we have
$3\binom{n-3}{3}$ ways of obtaining $Y_{u}^{s,t}$, the corresponding
covariance
terms contribute \eqref{33} to the bound.
\end{pf}

\begin{lemma}\label{lem12}
Let $(W,W')$ be the number of triangles in the exchangeable pair of
Erd\H{o}s--R\'enyi graphs
$(G,G')$ as defined above. If $Q_{-1}(G)=\IP[W'=W-1|G]$, then
\begin{eqnarray*}
\Var Q_{-1}(G) &\leq&\frac{(n-2)}{\binom{n}{2}}p^3(1-p)^2
\bigl(1-p^2\bigr)^{n-3} \bigl(1-p^3
\bigl(1-p^2\bigr)^{n-3} \bigr)
\\
&&{} +\frac{2(n-2)}{\binom{n}{2}}p^3(1-p)^2 \bigl(
\bigl(1-2p^2+p^3\bigr)^{n-3}-p^3
\bigl(1-p^2\bigr)^{2n-6} \bigr)
\\
&&{} +\frac{4\binom{n-2}{2}}{\binom{n}{2}}p^5(1-p)^2 \bigl((1-p)
\bigl(1-2p+p^3\bigr)^{n-4}-p\bigl(1-p^2
\bigr)^{2n-6} \bigr)
\\
&&{} +\frac{4\binom{n-2}{2}}{\binom{n}{2}}p^5(1-p)^2\bigl(1-p^2
\bigr)^{2n-8} \bigl(1-p-p\bigl(1-p^2\bigr)^{2}
\bigr)
\\
&&{} +\frac{12\binom{n-2}{3}}{\binom{n}{2}}p^6(1-p)^2 \bigl((1-p)^3
\bigl(1-2p+p^3\bigr)^{n-5}-\bigl(1-p^2
\bigr)^{2n-6} \bigr)
\\
&&{} +\frac{12\binom{n-2}{3}}{\binom{n}{2}}p^6(1-p)^2\bigl(1-p^2
\bigr)^{2n-9} \bigl((1-p)^2-\bigl(1-p^2
\bigr)^{3} \bigr)
\\
&&{} +\frac{3\binom{n-2}{3}}{\binom{n}{2}}p^6(1-p)^2\bigl(1-p^2
\bigr)^{2n-10} \bigl(4p^3-7p^4+4p^6-p^8
\bigr)
\\
&&{} +\frac{12\binom{n-2}{4}}{\binom{n}{2}}p^6(1-p)^2\bigl(1-p^2
\bigr)^{2n-10} \bigl(4p^3-7p^4+4p^6-p^8
\bigr).
\end{eqnarray*}
\end{lemma}

\begin{pf}
As in the proof of Lemma~\ref{lem11}, let $X_{i,j}$ be the indicator that
there
is an edge between vertices $i$ and $j$ and $V_i^{k,j}:=X_{i,j}X_{i,k}$
be the
indicator that there is a $V$-star on the vertices $\{i, j, k\}$ with elbow
$i$.
We easily find
\begin{eqnarray*}
Q_{-1}(G)= \frac{(1-p)}{\binom{n}{2}}\sum_{\{j,k\}}
\sum_{i\neq j,k}X_{j,k}V_i^{j,k}
\prod_{l\neq i,j,k}\bigl(1-V_l^{j,k}
\bigr),
\end{eqnarray*}
and from this point, the proof is very similar to the proof of Lemma~\ref{lem11}.
\end{pf}

We are now in the position to prove Lemma~\ref{lem10}.

\begin{pf*}{Proof of Lemma~\ref{lem10}}
Recall the notation of Lemmas \ref{lem11} and \ref{lem12}. We have the
following easy facts:
\begin{enumerate}[(iii)]
\item[(i)]$\IE W = \binom{n} {3}p^3$,

\item[(ii)]$\sigma_n^2:=\Var W = \binom {n} {3}
(p^3(1-p^3)+3(n-3)p^5(1-p)
)$,

\item[(iii)]$ q_1 = \IE Q_1(G) = (n-2)p^3(1-p)
(1-p^2)^{n-3}$;
\end{enumerate}
the second item yields the assertion about the rate of $\Var(W)$.
The first bound in \eqref{26} now follows from Theorem~\ref{thm3}
after noting
that
for $(W,W')$ the number of triangles in the exchangeable pair of Erd\H
{o}s--R\'enyi graphs
$(G,G')$ as defined above and $n^{\alpha}p\to c>0$ for $1/2\leq\alpha<1$,
then
\begin{eqnarray*}
\Var Q_1(G)= \bigo \bigl(p^5 \bigr), \qquad \Var
Q_{-1}(G) = \bigo \bigl(p^3/n \bigr),
\end{eqnarray*}
which follows easily from Lemmas \ref{lem11} and \ref{lem12} above.

In order to prove the second bound in \eqref{26}, we will apply Theorem~\ref{thm4}
with $G=G(n,p)$ an Erd\H{o}s--R\'enyi random graph, $G'$ obtained by
taking a step
from $G$
in
the Markov chain (reversible with respect to the law of $G(n,p)$) defined
previously, and $G''$ obtained as a step from $G'$ in the same Markov chain.
Setting $(W,W',W'')$ to be the number of triangles in the graphs $(G,G',G'')$,
and defining (as per Theorem~\ref{thm4})
\begin{eqnarray*}
Q_{i,i}(G)=\IP\bigl[W''=W'+i,
W'=W+i|G\bigr],
\end{eqnarray*}
it is easy to see that
\begin{eqnarray*}
Q_{1,1}(G)=\frac{p^2}{\binom{n}{2}^2}\sum_{\{j,k\}}
\sum_{i\neq
j,k}Y_{i}^{j,k} \sum
_{\{s,t\}\neq\{j,k\}}\sum_{u\neq s,t}Y_{u}^{s,t},
\end{eqnarray*}
where as in the proof of Lemma~\ref{lem11} we define
\begin{eqnarray*}
Y_{i}^{j,k}=(1-X_{j,k})X_{i,j}X_{i,k}
\prod_{l\neq i,j,k}(1-X_{l,j}X_{l,k}).
\end{eqnarray*}
From this point, we find
\begin{eqnarray*}
\IE\bigl\llvert Q_{1,1}(G)-Q_1(G)^2\bigr
\rrvert =\frac{p}{\binom{n}{2}}q_1,
\end{eqnarray*}
since for fixed $\{j,k\}$, only one of the set $\{Y_{i}^{j,k}\}
_{i=1}^n$ can
be
non-zero. A similar analysis shows
\begin{eqnarray*}
\IE\bigl\llvert Q_{-1,-1}(G)-Q_{-1}(G)^2\bigr
\rrvert =\frac
{1-p}{\binom{n}{2}}q_1
\end{eqnarray*}
and the second bound in \eqref{26} now follows from Theorem~\ref
{thm4} after
collecting the pertinent facts above.
\end{pf*}

\subsection{Embedded sum of independent random variables}\label{sec7}

We consider the case where $W$
has an embedded sum of independent random variables. This setting has
been the
most prominent way to prove LLTs by probabilistic arguments;
see, for example, Davis and McDonald \cite{Davis1995},
R{\"o}llin \cite{Rollin2005}, Barbour
\cite{Barbour2009},
Behrisch, Coja-Oghlan and Kang \cite{Behrisch2010a} and Penrose and Peres \cite{Penrose2010}. In this case, our theory
can be used to obtain bounds on the rates for an LLT using previously
established
bounds on rates of convergence in other metrics.

Let $W$ be an integer valued random variable with variance $\sigma^2$
and let
$\mC{F}$ be some $\sigma$-algebra. Assume that $W$ allows for a
decomposition of
the form
%
\begin{equation}
\label{34} W = Y+\sum_{i=1}^{N}
Z_i,
\end{equation}
where $N$ is $\mC{F}$-measurable, and where, conditional on $\mC{F}$, we have that
$Y,Z_1,\dots,Z_N$ are all independent of each other.
Note that in what follows, the distribution of $Y$ is not
relevant.

%
\begin{theorem}\label{thm12} Let $W=W(\sigma)$ be a family of integer valued
random variables satisfying \eqref{34} and with $\Var W = \sigma^2$.
Assume there are
constants $u$ and $\beta$, independent of $\sigma^2$, such that, conditional
on $\mC{F}$,
%
\begin{equation}
\label{35} 0< u \leq1-{\tfrac{1}{2}}D_1(Z_i)
\end{equation}
for all $1\leq i \leq N$, and such that
%
\begin{equation}
\label{36} \IP\bigl[N < \beta\sigma^2\bigr] = \bigo\bigl(
\sigma^{-k}\bigr)
\end{equation}
as $\sigma\toinf$ for some $k\geq2$. Then, with $\ww{W}
= (W-\IE W)/\sigma$, and as $\sigma\to\infty$,
\begin{eqnarray*}
\dloc \bigl(\law(W),\TP\bigl(\IE W,\sigma^2\bigr) \bigr) = \bigo
\biggl(\frac{\dk (\law(\ww{W}),\Phi )^{1-1/k}}{\sigma
} \biggr)
\end{eqnarray*}
and
\begin{eqnarray*}
\dloc \bigl(\law(W),\TP\bigl(\IE W,\sigma^2\bigr) \bigr) = \bigo
\biggl(\frac{\dw (\law(\ww{W}),\Phi
)^{1-2/(k+1)}}{\sigma} \biggr).
\end{eqnarray*}
Retaining the previous hypotheses, if \eqref{36} holds now for some
$k\geq1$,
then
\begin{eqnarray*}
\dtv \bigl(\law(W),\TP\bigl(\IE W,\sigma^2\bigr) \bigr) = \bigo
\bigl(\dw \bigl(\law(\ww{W}),\Phi \bigr)^{k/(k+1)} \bigr),
\end{eqnarray*}
as $\sigma\toinf$.
\end{theorem}

\begin{pf} First, consider the setup conditional on $\mC{F}$.
Divide the sum $Z_1+\cdots+Z_N$ into $k$ successive blocks,
each of size ${\lfloor N/k\rfloor}$, plus one last block with less
than ${\lfloor N/k\rfloor}$
elements. By
Lemma~\ref{lem5},
we have
\begin{eqnarray*}
D_k(W|\mC{F})\leq2 \prod_{l=1}^{k}
\eta_{N,l},
\end{eqnarray*}
where
\begin{eqnarray*}
\eta_{N,l} & =& D_1 (Z_{(l-1){\lfloor N/k\rfloor}+1}+\cdots+
Z_{l{\lfloor
N/k\rfloor} } )
\\
& \leq&\sqrt{\frac{8}{\uppi}} \Biggl(\frac{1}{4} + \sum
_{i=(l-1){\lfloor N/k\rfloor}+1}^{l{\lfloor N/k\rfloor} } \biggl(1-{\frac{1}{2}}D_1(Z_i)
\biggr) \Biggr)^{-1/2},
\end{eqnarray*}
where we have used Lemma~\ref{lem6}. Therefore, using assumption
\eqref
{35}, for
$l=1,\ldots,k$,
\begin{eqnarray*}
\eta_{N,l} \leq \biggl(\frac{8}{\uppi{\lfloor N/k\rfloor}u} \biggr)^{1/2}.
\end{eqnarray*}
Now, assume without loss of generality that $\sigma^2>k/\beta$.
In this case
\begin{eqnarray*}
\I\bigl[N\geq\beta\sigma^2\bigr]\frac{1}{{\lfloor N/k\rfloor}}\leq
\frac{1}{\beta\sigma^2/k-1}
\end{eqnarray*}
and since we can always trivially bound
$D_k(W|\mC{F})$ by $2^k$ because $\eta_{N,l}\leq2$, we have
\begin{eqnarray*}
D_k(W|\mC{F}) \leq2^k\I\bigl[N < \beta
\sigma^2\bigr] + \I\bigl[N \geq\beta\sigma^2\bigr] \biggl(
\frac{8k}{\uppi u
(\beta\sigma^2-k)} \biggr)^{k/2}.
\end{eqnarray*}
Hence, by Lemma~\ref{lem4},
\begin{eqnarray*}
D_k(W)\leq\IE D_k(W|\mC{F}) \leq2^k
\IP\bigl[N < \beta\sigma^2\bigr]+ \biggl(\frac{8k}{\uppi u (\beta\sigma^2-k)}
\biggr)^{k/2},
\end{eqnarray*}
which is $\bigo(\sigma^{-k})$ as $\sigma\toinf$. After noting
that $W$ is
integer valued and, hence, $\sigma^{-1}
=\bigo (\dw (\law(\ww{W}),\Phi )\wedge\dk (\law
(\ww{W}),\Phi ) )$, the
claims now follows easily from (ii), (iii), and (iv) of
Theorem~\ref{thm2}, and Lemma~\ref{lem7}, keeping in mind \eqref{225}.
\end{pf}

Note that, under the stated conditions, Theorem~\ref{thm12} implies
the LLT for
$W$ if it satisfies the CLT, as the latter also implies convergence in the
Kolmogorov metric. If a rate of convergence is available, Theorem~\ref{thm12}
also yields an upper bound on the rate of convergence for the LLT.

To illustrate Theorem~\ref{thm12}, we consider the so-called
\emph{independence number} of a random graph. The independence number
of a
graph $G$ is defined to be the maximal number of vertices that can be chosen
from the graph so that no two of these vertices are connected.

Consider the following random graph model, which is a
simplified version of one discussed by Penrose and Yukich \cite
{Penrose2005}. Let the
open set
$U\subset\IR^d$, $d\geq1$, be of finite volume, which, without loss of
generality, we assume to be $1$. Let $\mC{X}$ be a homogeneous Poisson point
process on $U$ with intensity $\lambda$ with respect to the Lebesgue measure.
Define $G(\mC{X},r)$ to be the graph on the vertex set $\mC{X}$ by connecting two
vertices whenever they are at most distant $r$ apart from each other.
In the
context of this random geometric graph, the independence number is the maximal
number of closed balls of radius $r/2$ with centers chosen from $\mC{X}$,
so that
no two balls are intersecting.

%
\begin{theorem}[(Penrose and Yukich \cite{Penrose2005})]\label
{thm13} For $b>0$, let $W_b$
be the
independence number in $G(\mC{X},b\lambda^{-1/d})$. Then, if $b$
is small enough, we have $\Var
W_b\asymp\lambda$ and
\begin{eqnarray*}
\dk \biggl(\law \biggl(\frac{W_b-\IE W_b}{\sqrt{\Var W_b}} \biggr), \N (0,1) \biggr) =\bigo
\bigl(\log(\lambda)^{3d}\lambda^{-1/2} \bigr)
\end{eqnarray*}
as $\lambda\toinf$.
\end{theorem}

The condition ``$b$ is small enough'' is described in greater detail
in Section~2.4 of Penrose and Yukich \cite{Penrose2005} and is
necessary to guarantee the
asymptotic order of the variance of $W_b$. We can give a local limit
result as follows.

%
\begin{theorem} Under the assumptions of Theorem~\ref{thm13}, we have
that for
every $\eps>0$,
\begin{eqnarray*}
\dloc \bigl(\law(W_b), \TP(\IE W_b,\Var
W_b) \bigr) =\bigo \bigl(\lambda^{-1+\eps} \bigr),
\end{eqnarray*}
as $\lambda\toinf$.
\end{theorem}

\begin{pf} Let $R=b\lambda^{-1/d}/2$. Denote by $B_R(x)$
the closed ball with radius $R$ and center $x$; define $\partial B_R(x) =
B_{2R}(x)\setminus B_{R}(x)$. Now, choose $n$ non-intersecting balls in
$U$, each of radius $3R$ and centers $x_1,\dots,x_n$; it is clear
that it is possible to have $n\asymp\lambda$. For ball $B_{3R}(x_i)$,
define the indicators
\begin{eqnarray*}
I_i = \I\bigl[\partial B_R(x_i)
\cap\mC{X}\mbox{ is empty}\bigr], \qquad J_i = \I\bigl[
B_R(x_i)\cap\mC{X}\mbox{ is not empty}\bigr].
\end{eqnarray*}
Note that the $I_i$ are independent and identically distributed and,
hence, $N = \sum_{i=1}^n I_i \sim\Bi(n,p)$ with $p = \IE I_i$ being bounded
away from $1$ and $0$ as $\lambda\toinf$. We let
$\mC{F}=\sigma(I_1,\dots,I_n)$. Furthermore, note that if $I_1=1$, then
$J_i$ is
exactly the contribution of the ball $B_{2R}(x_i)$ to the independence number
$W_b$, as within $B_R(x_i)$ all the vertices are connected and there is no
connection to any other vertices outside $B_{R}(x_i)$. Therefore we can find
$Y$ such that
\begin{eqnarray*}
W_n = Y + \sum_{j=1}^N
J_{K_j},
\end{eqnarray*}
where $K_1,\dots,K_N$ are the indices of those balls with $I_i=1$. Given
$\mC{F}$, note that $J_{K_1},\dots,J_{K_N}$ are independent $\Be(q)$,
with  $q = \IE J_{K_j}$ being bounded away from $0$ and
$1$, and they are also independent of $Y$. This implies condition
\eqref
{35} for
$u=q\wedge(1-q)$ which is bounded away from $0$ as $\lambda\toinf$. Using
usual exponential tail bounds for the binomial distribution, it is easy
to see that, for every $k$, one can find $\beta$ such that \eqref{36}
holds. In
combination with Theorem~\ref{thm13}, Theorem~\ref{thm12} now yields
the claim.
\end{pf}

\section*{Acknowledgements}

We are indebted to Steven Evans for pointing out that the inequalities we
proved and used in an earlier version of the manuscript were in fact
Landau--Kolmogorov inequalities, and thank Peter Eichelsbacher for
helpful discussions. We are also grateful to the two anonymous referees
for their helpful comments.

Both authors were partially supported by NUS research
grant R-155-000-098-133 and NR would like to express his gratitude for
the kind
hospitality of the Department of Statistics and Applied
Probability, NUS, during his research visit. AR was supported by NUS research
grant R-155-000-124-112.




\printhistory

\end{document}